\documentclass[smallcondensed,numbook]{svjour3}

\usepackage[T1]{fontenc}      
\usepackage[utf8]{inputenc}   
\usepackage{microtype}        

\usepackage{amsmath}
\usepackage{amssymb}
\usepackage{commath}          

\usepackage{mathtools}
\usepackage{thmtools}
\smartqed

\usepackage{enumerate} 
\usepackage{verbatim}               
\usepackage[dvipsnames]{xcolor}     

\usepackage[square,numbers]{natbib}
\allowdisplaybreaks

\DeclareUnicodeCharacter{00A0}{ }

\numberwithin{theorem}{section} 

\declaretheorem[numberlike=theorem,Refname={Lemma,Lemmas}]{lemma}

\declaretheorem[numberlike=theorem,Refname={Corollary,Corollaries}]{corollary}

\newcommand{\R}{\mathbb{R}}             
\newcommand{\N}{\mathbb{N}}             
\DeclareMathOperator{\neper}{e}         
\renewcommand{\epsilon}{\varepsilon}    
\newcommand{\rkhs}{\mathcal{H}}         

\DeclareMathOperator*{\argmin}{arg\,min}  

\DeclarePairedDelimiterX\Set[2]{\lbrace}{\rbrace}%
{ #1 \,:\, #2 } 

\DeclarePairedDelimiterX\inprod[2]{\langle}{\rangle}%
{ #1 , #2 } 
\newcommand{\T}{\mathsf{T}}         

\newcommand{\lspan}{\mathrm{span}}          

\renewcommand{\b}[1]{#1} 

\newcommand{\textsm}[1]{\text{\tiny{#1}}}

\begin{document}

\title{Worst-case optimal approximation with increasingly flat Gaussian kernels}
\author{Toni Karvonen \and Simo Särkkä}
\institute{
   Department of Electrical Engineering and Automation\\
  Aalto University, Espoo, Finland\\
  \email{tskarvon@iki.fi, simo.sarkka@aalto.fi}
}

\maketitle


\begin{abstract} We study worst-case optimal approximation of positive linear functionals in reproducing kernel Hilbert spaces induced by increasingly flat Gaussian kernels. This provides a new perspective and some generalisations to the problem of interpolation with increasingly flat radial basis functions. When the evaluation points are fixed and unisolvent, we show that the worst-case optimal method converges to a polynomial method. In an additional one-dimensional extension, we allow also the points to be selected optimally and show that in this case convergence is to the unique Gaussian quadrature type method that achieves the maximal polynomial degree of exactness. The proofs are based on an explicit characterisation of the reproducing kernel Hilbert space of the Gaussian kernel in terms of exponentially damped polynomials.
\end{abstract}

\keywords{Worst-case analysis \and Reproducing kernel Hilbert spaces \and Gaussian kernel \and Gaussian quadrature}


\section{Introduction}

Most popular kernels used scattered data approximation~\citep{FasshauerMcCourt2015,Wendland2005} and Gaussian process regression~\citep{RasmussenWilliams2006} are \emph{isotropic} (i.e., radial basis functions), depending only on the Euclidean distance $\norm[0]{\cdot}_2$ between the points:
\begin{equation} \label{eq:kernel}
  K_\ell(x,x') = \Phi \bigg( \frac{\norm[0]{x-x'}_2}{\ell} \bigg)
\end{equation}
for a continuous positive-definite function $\Phi \colon [0,\infty) \to \R$ and a \emph{length-scale} parameter $\ell > 0$. Given any function $f \colon \R^d \to \R$ evaluated at distinct points \sloppy{${X = \{x_1, \ldots, x_N\} \subset \R^d}$} such a kernel can be used to construct a unique \emph{kernel interpolant} based on the translates $\{K_\ell(\cdot, x_n)\}_{n=1}^N$. The kernel interpolant is
  \begin{equation} \label{eq:lagrange-form}
    s_{\ell,f,X}(x) = \sum_{n=1}^N f(x_n) u_{\ell,n}(x),
  \end{equation}
where $u_n$ are the Lagrange cardinal functions that solve the linear system
    \begin{equation} \label{eq:s-system}
      \begin{bmatrix} K_\ell(\b{x}_1, \b{x}_1) & \cdots & K_\ell(\b{x}_1, \b{x}_N) \\ \vdots & \ddots & \vdots \\ K_\ell(\b{x}_N,\b{x}_1) & \cdots & K_\ell(\b{x}_N, \b{x}_N) \end{bmatrix} \begin{bmatrix} u_{\ell,1}(x) \\ \vdots \\ u_{\ell,N}(x) \end{bmatrix} = \begin{bmatrix} K_\ell(x,x_1) \\ \vdots \\ K_\ell(x,x_N) \end{bmatrix}
    \end{equation}
and satisfy $u_{\ell,n}(x_m) = \delta_{nm}$.
Uniqueness of the solution for each $x \in \R^d$ is guaranteed by positive-definiteness of the matrix on the left-hand side of this system.

When $\ell \to \infty$, the kernel $K_\ell$ becomes increasingly flat and the linear system~\eqref{eq:s-system} increasingly ill-conditioned.\footnote{Note that most of the literature we cite parametrises the kernel in terms of the inverse length-scale $\epsilon = 1/\ell$ and accordingly considers the case $\epsilon \to 0$.} Nevertheless, the corresponding kernel interpolant is typically well-behaved at this limit. Starting with the work of \mbox{\citet{DriscollFornberg2002}}, it has been shown that a certain unisolvency assumption on $X$ implies that the kernel interpolant converges to (i) a polynomial interpolant if the kernel is infinitely smooth~\citep{DriscollFornberg2002,Fornberg2004,LarssonFornberg2005,LeeYoonYoon2007,Schaback2005,Schaback2008} or (ii) a polyharmonic spline interpolant if the kernel is finitely smooth~\citep{Lee2014,Song2012}. Further generalisations appear in~\citep{Lee2015}. The former case covers kernels such as Gaussians, multiquadrics, and inverse multiquadrics while the latter applies to, for example, Matérn kernels and Wendland's functions.
Among the most interesting of these results is the one by~\citet{Schaback2005} who proved that the interpolant at the increasingly flat limit of the Gaussian kernel
\begin{equation}\label{eq:gauss-kernel}
K_\ell(\b{x}, \b{x}') = \exp\bigg( \! - \frac{\norm[0]{\b{x}-\b{x}'}_2^2}{2\ell^2} \bigg) 
\end{equation}
exists regardless of the geometry of $X$ and coincides with the de Boor and Ron polynomial interpolant~\citep{deBoor1994,deBoorRon1992a}.
Furthermore, numerical ill-conditioning for large $\ell$, mentioned above, has necessitated the development of techniques for stable evaluation of the kernel interpolant~\citep{Cavoretto2015,FasshauerMcCourt2012,Fornberg2013,Wright2017}.
Increasingly flat kernels have been also discussed independently in the literature on the use of Gaussian processes for numerical integration~\cite{Minka2000,OHagan1991,Sarkka2016}, albeit accompanied only with non-rigorous arguments. Even though the intuition that the lowest degree terms in the Taylor expansion of the kernel dominate construction of the interpolant as $\ell \to \infty$ and that this ought to imply convergence to a polynomial interpolant is quite clear, this is not always translated into transparent proofs.

The purpose of this article is to generalise the aforementioned results on flat limits of kernel interpolants for worst-case optimal approximation of general positive linear functionals in the reproducing kernel Hilbert space (RKHS) of the Gaussian kernel~\eqref{eq:gauss-kernel}.
That such generalisations are possible is not perhaps surprising; it is rather the simple proof technique made possible by the worst-case framework and an explicit characterisation~\citep{Minh2010} of the Gaussian RKHS that we find the most interesting aspect of the present work.

\subsection{Worst-case optimal approximation}

Let $\Omega$ be a subset of $\R^d$ with a non-empty interior and $L \colon C(\Omega) \to \R$ a positive linear functional acting on continuous real-valued functions defined on $\Omega$ and satisfying \sloppy{${L[\abs[0]{p}] < \infty}$} for every polynomial $p$ on $\Omega$. 
The functionals most often discussed in this article are the point evaluation and the integration functionals
\begin{equation} \label{eq:functional-examples}
  L_x[f] = f(x) \quad \text{ and } \quad L_\mu[f] = \int_\Omega f \dif \mu \: \text{ for a Borel measure $\mu$ on $\Omega$},
\end{equation}
respectively.
Derivative evaluation functionals $L_x^{(n)}[f] = f^{(n)}(x)$ are also often considered.
A \emph{cubature rule} (\emph{quadrature} if $d=1$) $Q_X(w) \colon C(\Omega) \to \R$ with the distinct points $X = \{\b{x}_1, \ldots, \b{x}_N\} \subset \Omega$ and weights $\b{w} = (\b{w}(1), \ldots, \b{w}(N)) \in \R^N$ is a weighted approximation to $L$ of the form
\begin{equation}\label{eq:cubature}
Q_X(w)[f] = \sum_{n=1}^N \b{w}(n) f(\b{x}_n) \approx L[f].
\end{equation}
When restricted on $\Omega \times \Omega$, the positive-definite kernel $K_\ell$ in~\eqref{eq:kernel} induces a unique reproducing kernel Hilbert space $\rkhs(K_\ell) \subset C(\Omega)$ where the reproducing property $\inprod{f}{K_\ell(\cdot,x)}_{\mathcal{H}(K_\ell)} = f(x)$ holds for every $x \in \Omega$ and $f \in \mathcal{H}(K_\ell)$. With minor modifications everything in this section holds also when the kernel is not isotropic. Because the kernel is isotropic, $L[K_\ell(x,x)] \leq L[\Phi(0)] < \infty$ by the assumption that $L[p]$ is finite if $p$ is a polynomial.
This guarantees that $L[K_\ell(\cdot,x)] \in \rkhs(K_\ell)$ for any $x \in \Omega$ and consequently that $L[f] < \infty$ for any $f \in \rkhs(K_\ell)$. 

The \emph{worst-case error} $e_\ell(Q_X(w))$ of the cubature rule~\eqref{eq:cubature} in $\rkhs(K_\ell)$ is
\begin{equation}\label{eq:wce}
e_\ell\big(Q_X(w) \big) = \sup_{\norm[0]{f}_{\rkhs(K_\ell)} \leq 1} \, \abs[3]{L[f] - \sum_{n=1}^N \b{w}(n) f(\b{x}_n) }.
\end{equation}
Given a fixed set of distinct points, we are interested in the \emph{kernel cubature rule} $Q_X(w_\ell^*)$ whose weights are chosen so as to minimise the worst-case error:
\begin{equation*}
\b{w}_\ell^* = \argmin_{\b{w} \in \R^N} e_\ell \big(Q_X(w) \big) \quad \text{ and } \quad e_\ell(Q_X(w_\ell^*)\big) = \inf_{\b{w} \in \R^N} e_\ell \big( Q_X(w) \big).
\end{equation*}
These weights are unique and available as the solution to the linear system~\citep[Section~3.2]{Oettershagen2017}
\begin{equation}\label{eq:weight-system}
\begin{bmatrix} K_\ell(\b{x}_1, \b{x}_1) & \cdots & K_\ell(\b{x}_1, \b{x}_N) \\ \vdots & \ddots & \vdots \\ K_\ell(\b{x}_N,\b{x}_1) & \cdots & K_\ell(\b{x}_N, \b{x}_N) \end{bmatrix} \begin{bmatrix} \b{w}_\ell^*(1) \\ \vdots \\ \b{w}_\ell^*(N) \end{bmatrix} = \begin{bmatrix} L[K_\ell(\cdot, \b{x}_1)] \\ \vdots \\ L[K_\ell(\cdot, \b{x}_N)] \end{bmatrix}.
\end{equation}
Although our notation does not make this explicit, the weights obviously depend on the linear functional $L$ and the evaluation points $X$.
For each $x \in \R^d$, the kernel interpolant $s_{\ell,f,X}(x)$ now arises as the kernel cubature rule for approximation of the point evaluation functional $L_x$ in~\eqref{eq:functional-examples} and the Lagrange functions are $u_{\ell,n}(x) = w_\ell^*(n)$.
In this case the worst-case error coincides with the power function~\citep{Schaback1993}.
For an arbitrary $L$, the kernel cubature rule can be obtained by applying $L$ to the kernel interpolant:
\begin{equation*}
  Q_X(w_\ell^*) = L[s_{\ell,f,X}] = \sum_{n=1}^N f(x_n) L[u_{\ell,n}].
\end{equation*}
That is, the weights are $w_\ell^*(n) = L[u_{\ell,n}]$.

\subsection{Contributions}

Recall that we only consider the Gaussian kernel~\eqref{eq:gauss-kernel}. This article contains two theoretical main contributions:
\begin{itemize}
\item In Section~\ref{sec:fixed} we prove that if $X$ is unisolvent with respect to a full polynomial space $\Pi_m$ and $N = \dim \Pi_m$, then $Q_X(w_\ell^*)$ converges (as $\ell \to \infty$) to the unique cubature rule $Q_X(w_\textsm{pol})$ that satisfies \sloppy{${Q_X(w_\textsm{pol})[p] = L[p]}$} for every polynomial $p$ of degree at most $m$. This result, contained in Theorem~\ref{thm:unisolvent} and Corollary~\ref{cor:unisolvent-bounded}, is a generalisation for arbitrary positive linear functionals of the interpolation results cited earlier. If $\Omega$ is bounded, the results hold for any positive linear functional satisfying the mild assumptions imposed earlier. However, boundedness of $\Omega$ is not necessary: at the end of Section~\ref{sec:fixed} we supply an example involving integration over $\R^d$ with respect to the Gaussian measure.
\item In Section~\ref{sec:optimal} we present a generalisation, based on a theorem of \mbox{\citet{Barrow1978}}, for \emph{optimal kernel quadrature rules}~\cite[Chapter~5]{Oettershagen2017} that have both their points and weights selected so as to minimise the worst-case error. The result, Theorem~\ref{thm:optimal}, states that such rules, if unique, converge to the $N$-point Gaussian quadrature rule for the functional $L$, which is the unique quadrature rule $Q_{X_\textsm{G}}(w_\textsm{G})$ such that $Q_{X_\textsm{G}}(w_\textsm{G})[p] = L[p]$ for every polynomial $p$ of degree at most $2N-1$. This partially settles a conjecture posed by \citet[Section 3.3]{OHagan1991}, and further discussed in~\citep{Minka2000,Sarkka2016}, on convergence of optimal kernel quadrature rules to Gaussian quadrature rules.
\end{itemize}
Some generalisations for other kernels and cubature rules of more general form than~\eqref{eq:cubature} are briefly discussed in Section~\ref{sec:generalisations}.

\section{Fixed points} \label{sec:fixed}

The following theorem, which provides a characterisation of the RKHS of the Gaussian kernel~\eqref{eq:gauss-kernel}, is the central tool of this article. This results is due to \citet{Steinwart2006} and \citet{Minh2010}; see also~\citep[Section 4.4]{Steinwart2008} and~\citep[Example~3]{DeMarchiSchaback2009}. 
In this theorem (and the remainder of the article) $\N_0^d$ stands for the collection of $d$-dimensional non-negative multi-indices: $\N_0^d = \Set{ (\alpha_1, \ldots, \alpha_d) \in \R^d}{ \alpha_1, \ldots, \alpha_d \in \N_0}$. The absolute value and factorial of $\alpha \in \N_0^d$ are $\abs[0]{\alpha} = \alpha_1 + \cdots + \alpha_d$ and $\alpha! = \alpha_1! \times \cdots \times \alpha_d!$.

\begin{theorem}[Steinwart 2006; Minh 2010]\label{thm:gauss-rkhs} Let $\Omega$ be a subset of $\R^d$ with a non-empty interior. Then the RKHS $\rkhs(K_\ell)$ induced by the Gaussian kernel~\eqref{eq:gauss-kernel} with length-scale $\ell > 0$ consists of the functions
\begin{equation} \label{eq:RKHS-expansion}
f(\b{x}) = \neper^{-\norm[0]{\b{x}}_2^2/(2\ell^2)} \sum_{\b{\alpha} \in \N_0^d} f_{\b{\alpha}} \b{x}^{\b {\alpha}} \quad \text{such that} \quad \norm[0]{f}_{\rkhs(K_\ell)}^2 = \sum_{\b{\alpha} \in \N_0^d} \ell^{2\abs[0]{\b{\alpha}}} \b{\alpha}! f_{\b{\alpha}}^2 < \infty,
\end{equation}
where convergence is absolute. Its inner product is $\inprod{f}{g}_{\rkhs(K_\ell)} = \sum_{\alpha \in \N_0^d} \ell^{2\abs[0]{\alpha}} \alpha! f_\alpha g_\alpha$.
Furthermore, the collection
\begin{equation} \label{eq:orthobasis}
  \bigg\{ \frac{1}{\ell^{\abs[0]{\alpha}} \sqrt{\alpha!}} \neper^{-\norm[0]{x}_2^2/(2\ell^2)} x^\alpha \bigg\}_{\alpha \in \N_0^d}
\end{equation}
of functions forms an orthonormal basis of $\rkhs(K_\ell)$.
\end{theorem}

Two crucial implications of this theorem are that $\rkhs(K_\ell)$ consists of functions expressible as series of exponentially damped polynomials, the damping effect vanishing as $\ell \to \infty$, and that, due to the terms $\ell^{2\abs[0]{\b{\alpha}}}$ appearing in the RKHS norm, the high-degree terms contribute the most to the norm. Consequently, the worst-case error~\eqref{eq:wce}, taking into account only functions of at most unit norm, is dominated by low-degree terms when $\ell$ is large.
The rest of this section formalises this intuition.

Let $\Pi_m \subset C(\Omega)$ stand for the space of $d$-variate polynomials of degree at most $m \in \N_0$:
\begin{equation*}
\Pi_m = \mathrm{span} \Set{\b{x}^{\b{\alpha}}}{ \b{\alpha} \in \N_0^d,\: \abs[0]{\b{\alpha}} \leq m }.
\end{equation*}
In this section we assume that the point set $X \subset \Omega \subset \R^d$ is $\Pi_m$-\emph{unisolvent}. That is, 
\begin{equation*}
N = \# X = \dim \Pi_m = \binom{m+d}{d} = \frac{(m+d)!}{d! \, m!}
\end{equation*}
and the zero function is the only element of $\Pi_m$ that vanishes on $X$. This is equivalent to non-singularity of the (generalised) Vandermonde matrix
\begin{equation} \label{eq:vandermonde}
\b{P}_\Pi = \begin{bmatrix} \b{x}_1^{\b{\alpha}_1} & \cdots & \b{x}_1^{\b{\alpha}_N} \\ \vdots & \ddots & \vdots \\ \b{x}_N^{\b{\alpha}_1} & \cdots & \b{x}_N^{\b{\alpha}_N} \end{bmatrix} ,
\end{equation}
where $\{\b{\alpha}_1$, \ldots, $\b{\alpha}_N\} = \Set{\alpha \in \N_0^d}{\abs[0]{\alpha} \leq m} \subset \N_0^d$. It follows that there is a unique \emph{polynomial cubature rule} $Q_X(w_\textsm{pol})$ such that $Q_X(w_\textsm{pol})[p] = L[p] < \infty$ for every $p \in \Pi_m$. Its weights solve the linear system $\b{P}_\Pi^\T w_\textsm{pol} = \b{L}_\Pi$ of $N$ equations, where the $N$-vector $L_\Pi$ has the elements $[\b{L}_\Pi]_n = L[\b{x}^{\b{\alpha}_n}]$. 
In this section we prove that the worst-case optimal weights $w_\ell^*$ for the Gaussian kernel~\eqref{eq:gauss-kernel} converge to $w_\textsm{pol}$ as $\ell \to \infty$.

Define then
\begin{equation} \label{eq:phi-functions}
\phi_{\b{\alpha}}^\ell(\b{x}) = \neper^{-\norm[0]{\b{x}}_2^2/(2\ell^2)} \b{x}^{\b{\alpha}},
\end{equation}
so that functions in the Gaussian RKHS, characterised by Theorem~\ref{thm:gauss-rkhs}, are of the form $f(\b{x}) = \sum_{\b{\alpha} \in \N_0^d} f_{\b{\alpha}} \phi_{\b{\alpha}}^\ell(\b{x})$ for coefficients $f_{\b{\alpha}}$ decaying sufficiently fast. Since the exponential function has no real roots, determinant of the matrix
\begin{equation} \label{eq:vandermonde-general}
\b{P}_{\phi,\ell} = \begin{bmatrix} \phi_{\b{\alpha}_1}^\ell(\b{x}_1) & \cdots & \phi_{\b{\alpha}_N}^\ell(\b{x}_1) \\ \vdots & \ddots & \vdots \\ \phi_{\b{\alpha}_1}^\ell(\b{x}_N) & \cdots & \phi_{\b{\alpha}_N}^\ell(\b{x}_N) \end{bmatrix} 
\end{equation}
satisfies $\abs[0]{\b{P}_{\phi,\ell}} = \abs[0]{\b{P}_\Pi} \exp( - \sum_{n=1}^N \norm[0]{x_n}_2^2/(2\ell^2)) \neq 0$ and $P_{\phi,\ell}$ is hence non-singular. From non-singularity it follows that there are unique weights $\b{w}_{\phi,\ell}$ such that $Q_X(\b{w}_{\phi,\ell})[\phi_{\b{\alpha}}^\ell] = L[\phi_{\b{\alpha}}^\ell]$ for every $\b{\alpha} \in \N_0^d$ satisfying $\abs[0]{\b{\alpha}} \leq m$. The weights solve \sloppy{${\b{P}_{\phi,\ell}^\T \b{w}_{\phi,\ell} = \b{L}_{\phi,\ell}}$}, where the $N$-vector $L_{\Phi,\ell}$ has the elements $[\b{L}_{\phi,\ell}]_n = L[\phi_{\b{\alpha}_n}^\ell]$.\footnote{See~\citep{FasshauerMcCourt2012} for an interpolation method based on a closely related basis derived from a Mercer eigendecomposition of the Gaussian kernel and~\citep{Karvonen2018} for an explicit construction of weights similar to $w_{\phi,\ell}$ in the case $L$ is the Gaussian integral.} 
This auxiliary cubature rule plays an important role in our argument.
To summarise, the following three weights (or sequences of weights) appear in the proofs below:
  \begin{enumerate}
  \item The weights $w_\ell^*$, solved from~\eqref{eq:weight-system}, are the worst-case optimal weights for the Gaussian kernel~\eqref{eq:gauss-kernel}. The results concern the behaviour of these weights as $\ell \to \infty$.
  \item The weights $w_\textsm{pol}$ are constructed such that the cubature rule defined by them is exact for all polynomials up to degree $m$: $Q_X(w_\textsm{pol})[p] = L[p]$ whenever $p \in \Pi_m$.
  \item The auxiliary weights $w_{\phi,\ell}$ satisfy $Q_X(w_{\phi,\ell})[\phi^\ell_\alpha] = L[\phi_\alpha^\ell]$ for every $\ell > 0$ and $\abs[0]{\alpha} \leq m$.
  \end{enumerate}

\begin{lemma}\label{lemma:weight-sum} Suppose that $X$ is $\Pi_m$-unisolvent and $\lim_{\ell \to \infty} L[\phi_{\b{\alpha}}^\ell(x)] = L[x^\alpha]$ for every $\abs[0]{\b{\alpha}} \leq m$. Then there is a constant $C_{\ell_0} \geq 0$ such that \sloppy{${\sup_{\ell \geq \ell_0} \, \sum_{n=1}^N \abs[0]{\b{w}_{\phi,\ell}(n)} \leq C_{\ell_0}}$} for any $\ell_0 > 0$.
\end{lemma}
\begin{proof} The assumption $\lim_{\ell \to \infty} L[\phi_{\b{\alpha}}^\ell(x)] = L[x^\alpha]$ and unisolvency of $X$ imply that $\lim_{\ell \to \infty} \b{w}_{\phi,\ell} = \b{w}_\textsm{pol}$. Because $L[\abs[0]{p}] < \infty$ for any polynomial $p$, both the weights $w_\textsm{pol}$ and $w_{\phi,\ell}$ are finite, which implies the claim.
  \qed
\end{proof}

\begin{lemma} \label{lemma:convergence1} Suppose that $X$ is $\Pi_m$-unisolvent and $\lim_{\ell \to \infty} L[\phi_{\b{\alpha}}^\ell(x)] = L[x^\alpha]$ for every $\abs[0]{\b{\alpha}} \leq m$. If $(w_\ell)_{\ell > 0}$ is any sequence of weights such that
  \begin{equation*}
    \abs[1]{L[\phi_\alpha^\ell] - Q_X(w_\ell)[\phi_\alpha^\ell]} \to 0 \quad \text{ for every } \quad \abs[0]{\alpha} \leq m,
  \end{equation*}
  then $\lim_{\ell \to \infty} w_\ell = w_\textsm{pol}$.
\end{lemma}
\begin{proof} We have $P_\Pi^\T w_\Pi = L_\Pi$ and 
  \begin{equation*}
    \norm[0]{L_\Pi - P_\Pi^\T w_\ell}_2 \leq \norm[0]{L_\Pi - L_{\phi,\ell}}_2 + \norm[0]{L_{\phi,\ell} - P_{\phi,\ell}^\T w_\ell}_2 + \norm[0]{P_{\phi,\ell}^\T w_\ell - P_\Pi^\T w_\ell}_2,
  \end{equation*}
  where each of the terms on the right-hand side vanishes as $\ell \to \infty$. Because $\norm[0]{L_\Pi - P_\Pi^\T w_\ell}_2 = \norm[0]{P_\Pi^\T ( w_\textsm{pol} - w_\ell)}_2$ and $P_\Pi$ is non-singular, we conclude that $\lim_{\ell \to \infty} w_\ell = w_\textsm{pol}$.
  \qed
\end{proof}

We are ready to prove the main result of the article for a fixed $\Pi_m$-unisolvent point set $X \subset \Omega$ consisting of $N$ distinct points. First, by considering one of the basis functions~\eqref{eq:orthobasis} we show that $\abs[0]{L[\phi_\alpha^\ell] - Q_X(w_\ell^*)[\phi_\alpha^\ell]} \leq \sqrt{\alpha!} \ell^{\abs[0]{\alpha}} e_{\ell}(Q_X(w_\ell^*))$ for every $\alpha \in \N_0^d$. Second, the sub-optimal cubature rule $Q_X(w_{\phi,\ell})$ defined above can be used, in combination with~\eqref{eq:RKHS-expansion}, to establish the upper bound \sloppy{${e_{\ell}(Q_X(w_\ell^*)) \leq C \ell^{-(m+1)}}$}. These two bounds imply that $\abs[0]{L[\phi_\alpha^\ell] - Q_X(w_\ell^*)[\phi_\alpha^\ell]} \to 0$ for every $\abs[0]{\alpha} \leq m$. If $\lim_{\ell \to \infty} L[\phi_\alpha^\ell(x)] = L[x^\alpha]$, Lemma~\ref{lemma:convergence1} then implies that $w_\ell^* \to w_\textsm{pol}$.

\begin{theorem}\label{thm:unisolvent} Let $N = \dim \Pi_{m} $ for some $m \in \N_0$ and $X$ be $\Pi_{m}$-unisolvent. Suppose that $\lim_{\ell \to \infty}L[\phi_\alpha^\ell(x)] = L[x^\alpha]$ for every $\alpha \in \N_0^d$ such that $\abs[0]{\alpha} \leq m$ and that
\begin{equation}\label{eq:Lphi-growth}
L\Bigg[\sum_{ \abs[0]{\b{\alpha}} \geq m+1} \frac{\abs[0]{a_\alpha}}{\ell_0^{\abs[0]{\alpha}-(m+1)} \sqrt{\b{\alpha}!}} \abs[0]{x^\alpha} \Bigg] \leq C_L < \infty
\end{equation}
for some $\ell_0 > 1$ and any sequence $(a_\alpha)_{\alpha \in \N_0^d}$ such that $\sum_{\alpha \in \N_0^d} a_\alpha^2 \leq 1$. Then 
\begin{equation*}
\lim_{\ell \to \infty} \b{w}_\ell^* = \b{w}_\textsm{pol} \quad \text{ and } \quad e_{\ell}\big( Q_X(w_\ell^*) \big) = \mathcal{O}\big( \ell^{-(m+1)} \big),
\end{equation*}
where $w_\textsm{pol}$ are the weights of the unique polynomial cubature rule such that $Q_X(w_\textsm{pol})[p] = L[p]$ for every $p \in \Pi_m$.
\end{theorem}

\begin{proof} For every $\b{\alpha} \in \N_0^d$ select the function
\begin{equation*}
g_\alpha(\b{x}) = \frac{1}{\ell^{\abs[0]{\b{\alpha}}} \sqrt{\b{\alpha}!}} \neper^{-\norm[0]{\b{x}}_2^2/(2\ell^2)} \b{x}^{\b{\alpha}} = \frac{1}{\ell^{\abs[0]{\b{\alpha}}} \sqrt{\b{\alpha}!}} \phi_{\b{\alpha}}^\ell(\b{x}).
\end{equation*}
From Theorem~\ref{thm:gauss-rkhs} it follows that $\norm[0]{g_\alpha}_{\rkhs(K_\ell)}^2 = 1$ since $g_\alpha$ is one of the basis functions~\eqref{eq:orthobasis}. Thus, by definition of the worst-case error,
\begin{equation}\label{eq:lower}
\frac{1}{\ell^{\abs[0]{\b{\alpha}}} \sqrt{\b{\alpha}!}} \abs[1]{L[\phi_{\b{\alpha}}^\ell] - Q_X(w_\ell^*)[\phi_{\b{\alpha}}^\ell] } = \abs[1]{L[g_\alpha] - Q_X(w_\ell^*)[g_\alpha]} \leq e_{\ell}\big( Q_X(w_\ell^*) \big).
\end{equation}
Next we derive an appropriate upper bound on $e_{\ell}( Q_X(w_\ell^*) )$ by considering the unique sub-optimal cubature rule $Q_X(\b{w}_{\phi,\ell})$ that is exact for every $\phi_{\b{\alpha}}^\ell$ with $\abs[0]{\b{\alpha}} \leq m$. 
In the expansion~\eqref{eq:RKHS-expansion} of a function in $\rkhs(K_\ell)$ we have $L[\phi_\alpha^\ell] = Q_X(w_{\phi,\ell})[\phi_\alpha^\ell]$ for every term with $\abs[0]{\alpha} \leq m$. Consequently, the worst-case error admits the bound
\begin{equation*}\label{eq:upper-preliminary1}
\begin{split}
e_{\ell}\big( &Q_X(\b{w}_{\phi,\ell}) \big) \\
&= \sup_{\norm[0]{f}_{\rkhs(K_\ell)} \leq 1} \, \abs[4]{ L\Bigg[ \sum_{\abs[0]{\alpha} \geq m+1} f_\alpha \phi_\alpha^\ell \Bigg] - Q_X(w_{\phi,\ell})\Bigg[ \sum_{\abs[0]{\alpha} \geq m+1} f_\alpha \phi_\alpha^\ell \Bigg] } \\
&\leq \sup_{\norm[0]{f}_{\rkhs(K_\ell)} \leq 1} L\Bigg[ \sum_{\abs[0]{\alpha} \geq m+1} \abs[0]{f_\alpha} \abs[0]{\phi_\alpha^\ell} \Bigg] + \sup_{\norm[0]{f}_{\rkhs(K_\ell)} \leq 1} \, \abs[4]{ Q_X(w_{\phi,\ell})\Bigg[ \sum_{\abs[0]{\alpha} \geq m+1} f_\alpha \phi_\alpha^\ell \Bigg] },
\end{split}
\end{equation*}
where $f_{\b{\alpha}}$ are the coefficients that define $f \in \rkhs(K_\ell)$ in Theorem~\ref{thm:gauss-rkhs}.
A consequence of~\eqref{eq:RKHS-expansion} is that $\norm[0]{f}_{\rkhs(K_\ell)} \leq 1$ implies $\abs[0]{f_{\b{\alpha}}} \leq a_\alpha/(\ell^{\abs[0]{\b{\alpha}}} \sqrt{\b{\alpha}})$ for some real numbers $\abs[0]{a_\alpha} \leq 1$ such that $\sum_{\alpha \in \N_0^d} a_\alpha^2 \leq 1$. Therefore, for $\ell \geq \ell_0 > 1$,
\begin{equation*}\label{eq:upper-preliminary-L}
\begin{split}
\sup_{\norm[0]{f}_{\rkhs(K_\ell)} \leq 1} L\Bigg[ \sum_{\abs[0]{\alpha} \geq m+1} \abs[0]{f_\alpha} \abs[0]{\phi_\alpha^\ell} \Bigg] &\leq L\Bigg[\sum_{ \abs[0]{\b{\alpha}} \geq m+1} \frac{\abs[0]{a_\alpha}}{\ell^{\abs[0]{\b{\alpha}}} \sqrt{\b{\alpha}!}} \abs[0]{\phi_\alpha^\ell} \Bigg] \\
&\leq \ell^{-(m+1)} L\Bigg[\sum_{ \abs[0]{\b{\alpha}} \geq m+1} \frac{\abs[0]{a_\alpha}}{\ell_0^{\abs[0]{\alpha}-(m+1)} \sqrt{\b{\alpha}!}} \abs[0]{\phi_\alpha^\ell} \Bigg] \\
&\leq \ell^{-(m+1)} L\Bigg[\sum_{ \abs[0]{\b{\alpha}} \geq m+1} \frac{\abs[0]{a_\alpha}}{\ell_0^{\abs[0]{\alpha}-(m+1)} \sqrt{\b{\alpha}!}} \abs[0]{x^\alpha} \Bigg] \\
&\leq C_L \ell^{-(m+1)}
\end{split}
\end{equation*}
by assumption~\eqref{eq:Lphi-growth}. Moreover, because 
\begin{equation*}
\max_{n=1,\ldots,N} \, \abs[0]{\phi_{\b{\alpha}}^\ell(\b{x}_n)} \leq \max_{n=1,\ldots,N} \, \abs[0]{\b{x_n}^{\b{\alpha}}} \leq C_X  
\end{equation*}
for some $C_X > 0$ and every $\ell$, we have
\begin{align*}
\sup_{\norm[0]{f}_{\rkhs(K_\ell)} \leq 1} \, \abs[4]{ Q_X(w_{\phi,\ell})\Bigg[ \sum_{\abs[0]{\alpha} \geq m+1} f_\alpha \phi_\alpha^\ell \Bigg] } \hspace{-4cm} & \\
&\leq \sup_{\norm[0]{f}_{\rkhs(K_\ell)} \leq 1} \, \sum_{n=1}^N \abs[0]{ w_{\phi,\ell}(n) } \sum_{\abs[0]{\alpha} \geq m+1} \abs[0]{f_\alpha} \abs[0]{\phi_\alpha^\ell(x_n)} \\
&\leq \ell^{-(m+1)} \sum_{n=1}^N \abs[0]{ w_{\phi,\ell}(n) } \sum_{ \abs[0]{\b{\alpha}} \geq m+1} \frac{\abs[0]{a_\alpha}}{\ell^{\abs[0]{\alpha}-(m+1)} \sqrt{\b{\alpha}!}} \abs[0]{\phi_\alpha^\ell(x_n)} \\
&\leq \ell^{-(m+1)} \sum_{n=1}^N \abs[0]{ w_{\phi,\ell}(n) } \sum_{ \abs[0]{\b{\alpha}} \geq m+1} \frac{C_X}{\ell_0^{\abs[0]{\alpha}-(m+1)} \sqrt{\b{\alpha}!}} \\
&\leq \ell^{-(m+1)} \bigg( \sup_{\ell \geq \ell_0} \sum_{n=1}^N \abs[0]{ w_{\phi,\ell}(n) } \bigg) \sum_{ \abs[0]{\b{\alpha}} \geq m+1} \frac{C_X}{\ell_0^{\abs[0]{\alpha}-(m+1)} \sqrt{\b{\alpha}!}} \\
&\eqqcolon C_Q \ell^{-(m+1)}
\end{align*}
where $C_Q < \infty$ follows from convergence of the last term and Lemma~\ref{lemma:weight-sum}. 
Thus
\begin{equation} \label{eq:upper}
e_{\ell}\big(Q_X(w_{\phi,\ell})\big) \leq (C_L + C_Q) \ell^{-(m+1)} \eqqcolon C \ell^{-(m+1)}
\end{equation}
when $\ell \geq \ell_0$.
Since $Q_X(w_\ell^*)$ is worst-case optimal, we have thus established with~\eqref{eq:lower} and~\eqref{eq:upper} that, for sufficiently large $\ell$,
\begin{equation*}
\frac{1}{\ell^{\abs[0]{\b{\alpha}}} \sqrt{\b{\alpha}!}} \abs[1]{L[\phi_{\b{\alpha}}^\ell] - Q_X(w_\ell^*)[\phi_{\b{\alpha}}^\ell] } \leq e_{\ell}( Q_X(w_\ell^*) ) \leq e_{\ell}\big( Q_X(\b{w}_{\phi,\ell}) \big) \leq C \ell^{-(m+1)}
\end{equation*}
for every $\b{\alpha} \in \N_0^d$ such that $\abs[0]{\b{\alpha}} \leq m$ and a constant $C$ independent of $\ell$. That is,
\begin{equation}\label{eq:fixed-point-final}
\abs[1]{L[\phi_{\b{\alpha}}^\ell] - Q_X(w_\ell^*)[\phi_{\b{\alpha}}^\ell] } \leq C \sqrt{\alpha!} \, \ell^{-(m+1) + \abs[0]{\alpha}} \leq C \sqrt{m!} \, \ell^{-1} \to 0 \quad \text{ as } \quad \ell \to \infty.
\end{equation}
The claim then follows by setting $w_\ell = w_{\ell}^*$ in Lemma~\ref{lemma:convergence1}.
\qed
\end{proof}

Assumptions of Theorem~\ref{thm:unisolvent} hold, for instance, if the domain $\Omega$ is bounded.

\begin{corollary} \label{cor:unisolvent-bounded} Let $N = \dim \Pi_{m}$ for some $m \in \N_0$ and $X$ be $\Pi_{m}$-unisolvent. Suppose that $\Omega$ is bounded. Then 
\begin{equation*}
\lim_{\ell \to \infty} \b{w}_\ell^* = \b{w}_\textsm{pol} \quad \text{ and } \quad e_{\ell}\big( Q_X(w_\ell^*) \big) = \mathcal{O}\big( \ell^{-(m+1)} \big),
\end{equation*}
where $w_\textsm{pol}$ are the weights of the unique polynomial cubature rule such that $Q_X(w_\textsm{pol})[p] = L[p]$ for every $p \in \Pi_m$.
\end{corollary}
\begin{proof}
  On a bounded domain the convergence $\phi_\alpha^\ell(x) \to x^\alpha$ as $\ell \to \infty$ is uniform. Thus
  \begin{equation*}
    \abs[1]{L[x^\alpha] - L[\phi_\alpha^\ell]} \leq L[1] \sup_{x \in \Omega}\abs[0]{x^\alpha - \phi_\alpha^\ell(x)} \to 0
  \end{equation*}
  as $\ell \to \infty$ for every $\alpha \in \N_0^d$. Assumption \eqref{eq:Lphi-growth} is also satisfied:
\begin{equation*}
\begin{split}
L\Bigg[\sum_{ \abs[0]{\b{\alpha}} \geq m+1} \frac{\abs[0]{a_\alpha}}{\ell_0^{\abs[0]{\alpha}-(m+1)} \sqrt{\b{\alpha}!}} \abs[0]{x^\alpha} \Bigg] \leq L\Bigg[\sum_{ \abs[0]{\b{\alpha}} \geq m+1} \frac{\beta^\alpha}{\ell_0^{\abs[0]{\alpha}-(m+1)} \sqrt{\b{\alpha}!}} \Bigg] < \infty,
\end{split}
\end{equation*}
where $\beta = (b,\ldots,b) \in \R^d$ for $b = \sup_{z \in \Omega} \norm[0]{z}_2$ and finiteness follows from the assumption $L[1] < \infty$.
\qed
\end{proof}

However, boundedness of $\Omega$ is not necessary. Consider Gaussian integration:
\begin{equation*}
L[f] = \frac{1}{(2\pi)^{d/2}} \int_{\R^d} f(\b{x}) \neper^{-\norm[0]{x}_2^2/2} \dif \b{x} = \prod_{i=1}^d \Bigg[ \frac{1}{\sqrt{2\pi}} \int_\R f(\b{x}) \neper^{-x_i^2/2} \dif x_i \Bigg].
\end{equation*}
If $\alpha \in \N_0^d$ has an odd element, $L[\phi_\alpha^\ell] = L[x^\alpha] = 0$ for every $\ell > 0$ by symmetry. If $\alpha = 2\beta$ for some $\beta \in \N_0^d$ the convergence $L[\phi_\alpha^\ell(x)] \to L[x^\alpha]$ as $\ell \to \infty$ follows from the monotone convergence theorem.
To verify~\eqref{eq:Lphi-growth}, recall that the absolute moments of the standard Gaussian distribution are
\begin{equation*}
\begin{split}
L[\abs[0]{x^\alpha}] &= \pi^{-d/2} \prod_{i=1}^d 2^{\alpha_i/2} \Gamma\bigg(\frac{\alpha_i+1}{2}\bigg) \\
&= \Bigg[ \prod_{\alpha_i \text{ odd}} \pi^{-1/2} 2^{\alpha_i/2} \bigg(\frac{\alpha_i-1}{2}\bigg)! \Bigg] \times \Bigg[ \prod_{\alpha_i \text{ even}} (\alpha_i - 1)!! \Bigg],
\end{split}
\end{equation*}
where $\Gamma(\cdot)$ is the Gamma function. Because $(n-1)!! \leq \sqrt{n!}$ for any $n \in \N$ and 
\begin{equation*}
\pi^{-1/2} \frac{2^{n/2}}{\sqrt{n!}} \bigg( \frac{n-1}{2} \bigg)! = \pi^{-1/2} \frac{2^{n/2}}{\sqrt{n!}} \times \frac{(n-1)!!}{2^{(n-1)/2}} = \sqrt{\frac{2}{\pi}} \frac{(n-1)!!}{\sqrt{n!}} \leq \sqrt{\frac{2}{\pi}} \leq 1
\end{equation*}
if $n$ is odd, we have
\begin{equation*}
\frac{L[\abs[0]{x^\alpha}]}{\sqrt{\alpha!}} = \Bigg[ \prod_{\alpha_i \text{ odd}} \pi^{-1/2} \frac{2^{\alpha_i/2}}{\sqrt{\alpha_i!}} \bigg(\frac{\alpha_i-1}{2}\bigg)! \Bigg] \times \Bigg[ \prod_{\alpha_i \text{ even}} \frac{(\alpha_i - 1)!!}{\sqrt{\alpha_i!}} \Bigg] \leq 1.
\end{equation*}
Thus
\begin{equation*}
L\Bigg[\sum_{ \abs[0]{\b{\alpha}} \geq m+1} \frac{\abs[0]{a_\alpha}}{\ell_0^{\abs[0]{\alpha}-(m+1)} \sqrt{\b{\alpha}!}} \abs[0]{x^\alpha} \Bigg] \leq \sum_{ \abs[0]{\b{\alpha}} \geq m+1} \frac{1}{\ell_0^{\abs[0]{\alpha}-(m+1)} } < \infty
\end{equation*}
if $\ell_0 > 1$.

\section{Optimal points in one dimension} \label{sec:optimal}

Let $d=1$ and $\Omega = [a,b]$ for $a < b$. In this section we consider quadrature rules whose points are also selected so as to minimise the worst-case error. A kernel quadrature rule is \emph{optimal} if its points and weights satisfy
\begin{equation*}
e_\ell\big( Q_{X_\ell^*}(w_\ell^*) \big) = \inf_{ w \in \R^N, \, X \in \Omega^N} e_\ell\big( Q_X(w) \big).
\end{equation*}
In order to eliminate degrees of freedom in ordering the points we require that the points are in ascending order (i.e., $x_n \leq x_{n+1}$). Even though optimal kernel quadrature rules have been studied since the 1970s~\citep{BarrarLoebWerner1974,Bojanov1979,Larkin1970,Richter1970,RichterDyn1971a} for the integration functional $L[f] = \int_a^b f(x) \omega(x) \dif x$, $\omega(x) > 0$, their theory is still not complete (the main results have been recently collated by \citet[Section 5.1]{Oettershagen2017}). 
Although uniqueness results are been proved only for totally positive isotropic kernels of the form~\eqref{eq:kernel} and integration when $\omega \equiv 1$~\citep{BraessDyn1982}, there exists numerical evidence suggesting that the optimal rule is unique in more general settings~\citep[p.\@~97]{Oettershagen2017}.
Note that the Gaussian kernel~\eqref{eq:gauss-kernel} we consider is totally positive.

In Theorem~\ref{thm:optimal} we show that uniqueness of an optimal kernel quadrature rule for each $\ell > 0$ implies that its increasingly flat limit is $Q_\textsm{G} = Q_{X_\textsm{G}}(w_\textsm{G})$, the $N$-point \emph{Gaussian quadrature rule} for the linear functional $L$. This is the unique quadrature rule that is exact for every polynomial of degree at most $2N-1$: $Q_\textsm{G}[x^n] = L[x^n]$ whenever $n \leq 2N-1$. This degree of exactness is maximal; there are no $N$-point quadrature rules exact for all polynomials up to degree $2N$. The most familiar methods of this type are of course the classical Gaussian quadrature rules for numerical integration~\citep[Section~1.4]{Gautschi2004}. For example, the Gauss--Legendre quadrature rule satisfies
\begin{equation*}
  Q_\textsm{G}[p] = \int_{-1}^1 p(x) \dif x
\end{equation*}
for every polynomial $p$ of degree at most $2N-1$ and its points are the roots of the $N$th degree Legendre polynomial.
Theorem~\ref{thm:optimal} was conjectured by \mbox{\citet[Section 3.3]{OHagan1991}} in 1991 in the form that the optimal kernel quadrature rule has the classical Gauss--Hermite quadrature rule as its increasingly flat limit if the kernel is Gaussian and $L$ is the Gaussian integral.
More discussion of this conjecture---but no rigorous proofs---can be found in~\citep[Section~4]{Minka2000}.

The proof of Theorem~\ref{thm:optimal} is based on a general result by \citet{Barrow1978} on existence and uniqueness of \emph{generalised Gaussian quadrature rules}. This result replaces the polynomials in a Gaussian quadrature rule with \emph{generalised polynomials} formed out of functions that constitute an \emph{extended Chebyshev system}~\citep[Chapter~1]{KarlinStudden1966}. A collection $\{u_n\}_{n=0}^{m-1} \subset C^{m-1}([a,b])$ of functions is an extended Chebyshev system if any non-trivial linear combination of the functions has at most $m-1$ zeroes, counting multiplicities. That is, if $u \in \lspan(\{u_n\}_{n=0}^{m-1})$ and $u^{(q_p)}(x_p) = 0$ for $x_p \in [a,b]$, $p = 1,\ldots,P$, and $q_p = 0,\ldots,Q_p-1$, then $\sum_{p=1}^P Q_p \leq m-1$. Any basis of the space of polynomials of degree at most $m-1$ is an extended Chebyshev system. Importantly, the functions $\{\phi_n^\ell\}_{n=0}^{m-1}$ in~\eqref{eq:phi-functions} are an extended Chebyshev system for any $m \in \N$. To verify this, note that any $\phi \in \lspan(\{\phi_n^\ell\}_{n=0}^{m-1})$ can be written as $\phi(x) = \neper^{-x^2/(2\ell^2)} p(x)$ for some polynomial $p$ of degree at most $m-1$ and consequently
\begin{equation*}
  \phi^{(l)}(x) = \neper^{-x^2/(2\ell^2)} \bigg(\sum_{r=0}^{l-1} s_r(x) p^{(r)}(x) + p^{(l)}(x)\bigg)
\end{equation*}
for some polynomials $s_r$. From this expression we see that $\phi^{(l)}(x) = 0$ for every $l=0,\ldots,q$ if and only if $p^{(l)}(x) = 0$ for every $l=0,\ldots,q$. Since $p$ can have at most $m-1$ zeroes, counting multiplicities, it follows that the same is true of $\phi$.

\begin{theorem}[Barrow 1978] \label{thm:barrow} Let $\{u_n\}_{n=0}^{2N-1} \subset C^{2N-1}([a,b])$ be an extended Chebyshev system and $L$ a positive linear functional on $\lspan(\{u_n\}_{n=0}^{2N-1})$. Then there exist unique points \sloppy{${a < x_1 < \cdots < x_N < b}$} and positive weights $w \in \R_+^N$ such that
  \begin{equation*}
    Q_X(w)[u_n] = L[u_n] \quad \text{ for every } \quad n = 0, \ldots, 2N-1.
  \end{equation*}
\end{theorem}

\begin{lemma}\label{lemma:weight-sum-2} 
Let $\Omega \subset \R^d$ and suppose that a cubature rule $Q_X(w)$ with non-negative weights satisfies $Q_X(w)[u] = L[u]$ for some positive function \sloppy{${u \colon \Omega \to (0, \infty)}$} such that $0 < c_l \leq u(x) \leq c_u$ for all $x \in \Omega$. Then
  \begin{equation*}
    \max_{n = 1,\ldots,N} w(n) \leq \sum_{n=1}^N w(n) \leq L[1] \frac{c_u }{ c_l }.
  \end{equation*}
\end{lemma}
\begin{proof} The claim follows immediately from the inequalities
  \begin{equation*}
   c_l \sum_{n=1}^N w(n) \leq \inf_{ x \in \Omega} u(x) \sum_{n=1}^N w(n) \leq \sum_{n=1}^N w(n) u(x_n) = L[u] \leq L[1] c_u . 
  \end{equation*}
  \qed
\end{proof}

\begin{lemma} \label{lemma:convergence2} Let $A$ be a metric space, $\ell_0 > 0$ a constant, and \sloppy{${g \colon [\ell_0, \infty) \times A \to [0,\infty)}$} a function. If there is a continuous function \sloppy{${g_\infty \colon A \to [0, \infty)}$} such that \sloppy{${g(\ell, \cdot) \to g_\infty}$} uniformly as $\ell \to \infty$ and a unique minimiser $x_\infty^*$ for which $g_\infty(x_\infty^*) = 0$, then a function $z \colon [\ell_0, \infty) \to A$ such that $\lim_{\ell \to \infty} g(\ell, z(\ell)) = 0$ has $\lim_{\ell \to \infty} z(\ell) = x_\infty^*$.
\end{lemma}
\begin{proof}
  The inequality $g_\infty(z(\ell)) \leq g(\ell, z(\ell)) + \abs[0]{ g_\infty(z(\ell)) - g(\ell, z(\ell)) }$ shows that $g_\infty(z(\ell)) \to 0$ since $g(\ell, z(\ell)) \to 0$ by assumption and $\abs[0]{ g_\infty(z(\ell)) - g(\ell, z(\ell)) } \to 0$ by uniformity of the convergence $g(\ell, \cdot) \to g_\infty$. Because $g_\infty$ is continuous, non-negative, and has a unique minimiser $x_\infty^*$, this implies that $z(\ell) \to x_\infty^*$.
  \qed
\end{proof}

\begin{theorem} \label{thm:optimal} Suppose that $\Omega = [a,b]$ for $a < b$. If for every $\ell > 0$ there exists a unique optimal kernel quadrature rule $Q_\ell^* = Q_{X_\ell^*}(w_\ell^*)$, then its points and weights converge to those of the $N$-point Gaussian quadrature rule for $L$:
  \begin{equation*}
    \lim_{\ell \to \infty} X_\ell^* = X_\textsm{G} \quad \text{ and } \quad \lim_{\ell \to \infty} w_\ell^* = w_\textsm{G},
  \end{equation*}
  where $X_\textsm{G}$ and $w_\textsm{G}$ are the unique points and weights such that $Q_{X_\textsm{G}}(w_\textsm{G})[x^n] = L[x^n]$ for every $0 \leq n \leq 2N-1$.
Moreover, $e_{\ell}(Q_\ell^*) = \mathcal{O}(\ell^{-2N})$.
\end{theorem}

\begin{proof} In a manner identical to the proof of Theorem~\ref{thm:unisolvent}, we establish the lower bound
\begin{equation*}
\frac{1}{\ell^{n} \sqrt{n!}} \abs[1]{L[\phi_n^\ell] - Q_\ell^*[\phi_n^\ell]} \leq e_{\ell}(Q_{\ell}^*)
\end{equation*}
that holds for every $n \geq 0$. Because $\{\phi_n^\ell\}_{n=0}^{2N-1}$ are an extended Chebyshev system, Theorem~\ref{thm:barrow} guarantees the existence of a unique $N$-point quadrature rule \sloppy{${Q_\textsm{G}^\ell = Q_{X_\textsm{G}^\ell}( w_\textsm{G}^\ell)}$} such that $Q_\textsm{G}^\ell[\phi_n^\ell] = L[\phi_n^\ell]$ for every $n \leq 2N-1$. The points $X_\textsm{G}^\ell = \{x_1^{\textsm{G},\ell}, \ldots x_N^{\textsm{G},\ell}\}$ of this rule are distinct and lie inside $\Omega$ and the weights $w_\textsm{G}^\ell$ are positive. We can then replicate the rest of the proof of Theorem~\ref{thm:unisolvent} in one dimension but with $m=2N-1$ and Lemma~\ref{lemma:weight-sum} replaced with Lemma~\ref{lemma:weight-sum-2} (applied to the function $u = \phi_0^\ell$) to show that, for sufficiently large $\ell$ and a constant $C$ independent of $\ell$,
\begin{equation*}
\frac{1}{\ell^n \sqrt{n!}} \abs[1]{L[\phi_n^\ell] - Q_{\ell}^*[\phi_n^\ell]} \leq e_{\ell}(Q_{\ell}^*) \leq e_{\ell} (Q_\textsm{G}^\ell) \leq C \ell^{-2N}
\end{equation*}
for every $n \leq 2N-1$. Consequently,
\begin{equation} \label{eq:optimal-conv}
\abs[1]{L[\phi_n^\ell] - Q_{\ell}^*[\phi_n^\ell]} \leq C \sqrt{n!} \, \ell^{n-2N} \leq C \sqrt{(2N-1)!} \, \ell^{-1} \to 0 \quad \text{ as } \quad \ell \to \infty
\end{equation}
for every $n \leq 2N-1$. We then fix $\ell_0 > 0$ and invoke Lemma~\ref{lemma:convergence2} with the function
\begin{equation*}
g\big(\ell, (X,w) \big) = \sum_{ n = 0}^{2N-1} \abs[1]{ L[\phi_n^\ell] - Q_X(w)[\phi_n^\ell] },
\end{equation*}
domain $A = (\Omega^N \times [0, \infty)^N)$, and $z(\ell) = (X_\ell^*, w_{\ell}^*)$.
Because the domain $\Omega = [a,b]$ is bounded, $\lim_{\ell \to \infty} L[\phi_n^\ell] \to L[x^n]$ for every $n \in \N_0$. Thus
\begin{equation*}
  g\big(\ell, (X,w) \big) \to g_\infty((X,w)) \coloneqq \sum_{ n = 0}^{2N-1}  \abs[1]{ L[x^n] - Q_X(w)[x^n] } \quad \text{ as } \quad \ell \to \infty
\end{equation*}
uniformly on $A$. Since the unique minimiser of $g_\infty$ is $(X_\textsm{G}, w_\textsm{G})$, the claim follows from~\eqref{eq:optimal-conv} and Lemma~\ref{lemma:convergence2}.
\qed
\end{proof}

\section{Generalisations} \label{sec:generalisations}

This section discusses some straightforward generalisations of the results in Sections~\ref{sec:fixed} and~\ref{sec:optimal}.

\subsection{Damped power series kernels}

Theorem~\ref{thm:gauss-rkhs} for the Gaussian kernel~\eqref{eq:gauss-kernel} is a consequence of the identity
\begin{equation*}
  \begin{split}
    K_\ell(x,x') &= \neper^{-\norm[0]{x}^2/(2\ell^2)} \neper^{-\norm[0]{x'}^2/(2\ell^2)} \sum_{ \alpha \in \N_0^d} \frac{1}{\alpha! \ell^{2\abs[0]{\alpha}}} x^\alpha (x')^\alpha \\
    &\eqqcolon \neper^{-\norm[0]{x}^2/(2\ell^2)} \neper^{-\norm[0]{x'}^2/(2\ell^2)} K_\ell^\textsm{pow}(x,x'),
    \end{split}
\end{equation*}
where $K^\textsm{pow}_\ell(x,x')$ is a \emph{power series kernel}~\citep{Zwicknagl2009}. Accordingly, the results in Sections~\ref{sec:fixed} and~\ref{sec:optimal} can be generalised for a class of kernels that we call \emph{damped power series kernels}. Let $G \colon \R^d \to \R \setminus \{0\}$ be a non-zero function and define $G_\ell(x) = G(\norm[0]{x}/\ell)$. Then a damped power series kernel is
\begin{equation} \label{eq:damped-power-series}
  K_\ell(x,x') = G_\ell(x) G_\ell(x') \sum_{ \alpha \in \N_0^d} \frac{\omega_\alpha}{(\alpha!)^2 \ell^{q \abs[0]{\alpha}}} x^\alpha (x')^\alpha
\end{equation}
for $q > 0$ and weight parameters $\omega_\alpha > 0$ such that the series converges for any $\ell > 0$ and $x,x' \in \Omega$. Arguments identical to those used in~\citep{Minh2010,Zwicknagl2009} establish that $K_\ell$ is a positive-definite kernel and that its RKHS $\rkhs(K_\ell)$ consists of functions
\begin{equation*}
  f(x) = G_\ell(x) \sum_{ \alpha \in \N_0^d} f_\alpha x^\alpha \quad \text{ such that } \quad \norm[0]{f}_{\rkhs(K_\ell)}^2 = \sum_{ \alpha \in \N_0^d} \frac{(\alpha!)^2 \ell^{q \abs[0]{\alpha}}}{\omega_\alpha} f_\alpha^2 < \infty.
\end{equation*}
The Gaussian kernel is recovered by setting $G(x) = \neper^{-\norm[0]{x}_2^2/2}$, $q = 2$, and $\omega_\alpha = \alpha!$. Note that the Gaussian kernel is an exception; damped power series kernels are rarely stationary.

Denote $\psi_\alpha^\ell(x) = G_\ell(x) x^\alpha$. If we assume that (i) $G$ is bounded, (ii) \sloppy{${\lim_{\ell \to \infty} L[ \psi_\alpha^\ell(x) ] \to L[x^\alpha]}$} for every $\alpha \in \N_0^d$, and (iii) a summability condition analogous to~\eqref{eq:Lphi-growth} holds, then a generalisation for damped power series kernels of Theorem~\ref{thm:unisolvent} is readily obtained. To generalise Theorem~\ref{thm:optimal} we also need to assume that $\{\psi_n\}_{n=0}^{2N-1}$ constitutes an extended Chebyshev system.

\subsection{Taylor space kernels}

Let $d = 1$. \emph{Taylor space kernels}~\citep{Dick2006,ZwicknaglSchaback2013} are obtained by selecting $G \equiv 1$ in~\eqref{eq:damped-power-series}.
As $\ell \to \infty$, the corresponding kernel quadrature rules then converge to polynomial rules. Perhaps the two most interesting special cases are the exponential kernel
\begin{equation*}
  K_\ell(x,x') = \exp\bigg(\frac{xx'}{\ell}\bigg) = \sum_{n=0}^\infty \frac{(xx')^n}{\ell^n n!}
\end{equation*}
and the Szeg\H{o} kernel
\begin{equation*}
  K_\ell(x,x') = \frac{\ell^2}{\ell^2 - xx'} = \frac{1}{1 - \ell^{-2} xx'} = \sum_{n=0}^\infty \ell^{-2n} (xx')^n.
\end{equation*}
The Szeg\H{o} kernel induces a Hardy space on a disk of radius $\ell$. Interestingly, it has been pointed out already in the 1970s that approximation with the Szeg\H{o} kernel yields polynomial methods as $\ell \to \infty$~\citep[Section~3]{Larkin1970}. See also~\citep[Section 4]{Minka2000}. An extensive numerical investigation has been recently published by \citet[Section~6.2]{Oettershagen2017}.

\subsection{General information functionals}

It would also be easy to replace the cubature rule~\eqref{eq:cubature} with a generalised version
\begin{equation*}
  Q[f] = \sum_{n=1}^N w(n) L_n[f],
\end{equation*}
where $L_n$ are any bounded linear functionals. If $L_n$ are such that the matrices
\begin{equation*}
  \begin{bmatrix} L_1[x^{\alpha_1}] & \cdots & L_1[x^{\alpha_N}] \\ \vdots & \ddots & \vdots \\ L_N[x^{\alpha_1}] & \cdots & L_N[x^{\alpha_N}] \end{bmatrix} \quad \text{ and } \quad \begin{bmatrix} L_1[\phi_{\alpha_1}^\ell(x)] & \cdots & L_1[\phi_{\alpha_N}^\ell(x)] \\ \vdots & \ddots & \vdots \\ L_N[\phi_{\alpha_1}^\ell(x)] & \cdots & L_N[\phi_{\alpha_N}^\ell(x)] \end{bmatrix},
\end{equation*}
which are generalisations of~\eqref{eq:vandermonde} and~\eqref{eq:vandermonde-general}, are non-singular, then Theorem~\ref{thm:unisolvent} and Corollary~\ref{cor:unisolvent-bounded} can be generalised.

\subsection{Non-unisolvent point sets}

If the kernel is Gaussian but point set $X \subset \Omega$ is not unisolvent, \citet{Schaback2005} has proved that the kernel interpolant~\eqref{eq:lagrange-form} converges the de Boor and Ron polynomial interpolant~\citep{deBoor1994,deBoorRon1992a}, which is the unique interpolant to $f$ at $X$ in a point-dependent polynomial space $\Pi_X$ having in a certain sense minimal degree.
  We expect that extensions for non-unisolvent points of the results in Section~\ref{sec:fixed} are possible.
  The kernel cubature weights would presumably convergence to the weights $w_\textsm{pol}'$ such that $Q_X(w_\textsm{pol}')[p] = L[p]$ for every $p \in \Pi_X$.

\begin{acknowledgements}
This work was supported by the Aalto ELEC Doctoral School and the Academy of Finland.
We thank the reviewers for numerous comments that helped in improving the presentation.
\end{acknowledgements}


\end{document}